\documentclass[11pt]{article}   	
\usepackage{geometry}                		
\usepackage{graphicx}				
\usepackage{amsmath, amsthm, amssymb}

\newtheorem{thm}{Theorem}
\newtheorem{lem}[thm]{Lemma} 
\newtheorem{definition}[thm]{Definition}
\newtheorem{prop}[thm]{Proposition} 
\newtheorem{cor}[thm]{Corollary} 

\newtheorem{question}[thm]{Question}

\def\eps{\epsilon}

\def\N{\mathbb N}

\def\R{\mathbb R}

\def\cC{\mathcal C}
\def\cU{\mathcal U}
\def\cV{\mathcal V}

\title{Rado's paracompactness theorem for conformal manifolds}
\author{Michael Kapovich}

\begin{document}
\maketitle


In this note, an ($n$-dimensional) manifold is a Hausdorff topological space equipped with a maximal smooth atlas with values in $\R^n$, 
no paracompactness is assumed. Recall that a {\em conformal structure} on a manifold $M$ is a reduction of the structure group of $TM$ from $GL(n,\R)$ to $CO(n)=\mathbb R_+\times O(n)$. Equivalently, a conformal structure is a collection of locally defined Riemannian metrics on $M$ which are conformal to each other. We will refer to such locally defined Riemannian metrics as {\em local conformal metrics} on $M$. Once we know that $M$ is paracompact, the structure group can be further reduced 
to $O(n)$ and local conformal metrics can be replaced by a single Riemannian metric on the entire $M$ inducing the conformal structure. 
A {\em conformal manifold} is a manifold equipped with a conformal structure. We will suppress the notation for a conformal structure and denote 
conformal manifolds by a single letter, e.g., $M$.  We will prove: 

\begin{thm}\label{thm:main}
Every conformal manifold $M$ of dimension $n\ge 2$ is paracompact.\footnote{Note that this result obviously fails for 1-dimensional manifolds.} 
\end{thm}

This result is known as {\em Rado's theorem} for $n=2$ and oriented conformal manifolds, equivalently, for Riemann surfaces 
(see e.g. \cite[\S 23]{Forster} for a proof using Perron's method). 
It is well-known that paracompactness fails for complex manifolds of complex dimension $\ge 2$ 
(examples are due to Calabi and Rosenlicht, \cite{CR}). 
The goal of this note is to show that Rado's theorem is actually a 
theorem of conformal, rather than complex, geometry.  

\begin{question}
What are other differential-geometric structures on manifolds which imply paracompactness? 
\end{question}

For instance, the existence of a symplectic structure is not an obstruction to paracompactness, as one can take, for instance, 
 the canonical symplectic form on the cotangent bundle of a non-paracompact manifold. 
 
 \medskip
{\bf Acknowledgements.} This work was motivated by a discussion of Perron's method and Rado's theorem with Bernhard Leeb. 

\section{Topological preliminaries}

A collection $\cC$ of subsets  of a topological space $X$ is said to be {\em locally finite} if every $x\in X$ has a neighborhood which intersects only 
finitely many members of $\cC$. 
A topological space $X$ is called {\em paracompact} if every open cover of $X$ admits a locally finite open refinement. (Unlike in the definition of 
compactness, a {\em refinement} cannot be replaced by a {\em subcover}.) This notion was introduced by Dieudonn\'e in \cite{D}. 

\begin{thm}
(See e.g. \cite[Theorem 5.1.3]{Engelking}.)  
Every metrizable topological space is paracompact. 
\end{thm}

This theorem has a ``converse'' of sorts:

\begin{thm}
 (Smirnov's theorem, see e.g. \cite[5.4.A]{Engelking}.) 
Every locally metrizable paracompact space is metrizable. 
\end{thm}

In general, paracompactness is not a hereditary property. However:  

\begin{lem}
Let $X$ be a locally metrizable space (e.g. a manifold). Then every paracompact subset $Y$ of $X$ is hereditarily paracompact. 
\end{lem}
\begin{proof} Since $X$ is locally metrizable, so is $Y$, hence, by Smirnov's theorem, $Y$ is metrizable. 
This implies that every subset $Z\subset Y$ is also metrizable, hence, paracompact. 
\end{proof}

\begin{thm}
(See e.g. \cite[Theorem 5.1.34]{Engelking}.) Suppose that $X$ is a topological space which is a union of a 
locally finite family of closed paracompact subsets. Then $X$ is itself paracompact. 
\end{thm}

\begin{cor}\label{cor:CA1}
If $X$ is a union of finitely many closed paracompact subsets, then $X$ is paracompact. 
\end{cor}

Recall that a space is said to be $\sigma$-compact if it is a union of countably many compact subsets. The following theorem was first proven by  
Dieudonn\'e in \cite{D}. For a textbook reference, see \cite[Ch. I.10, Theorem 5]{Bourbaki}.

\begin{thm}\label{prop:PA1}
Suppose that $X$ is a locally compact Hausdorff $\sigma$-compact space. Then $X$ is paracompact. 
Conversely, a locally compact Hausdorff space $X$ is paracompact if and only if it is the coproduct of a family of (pairwise disjoint) subspaces $X_i, i\in J$, 
each of which is Hausdorff, locally compact and $\sigma$-compact. 
\end{thm}

\begin{cor}
A manifold is paracompact if and only if every connected component is second countable. 
\end{cor}

\begin{lem}\label{lem:L1}
Every compact subset $K$ of a connected manifold $M$ is contained in an open (in $M$) connected relatively compact (hence, paracompact) submanifold $N\subset M$.
\end{lem}
\begin{proof} First, let $\cU$ be a finite cover of $K$ by open coordinate balls $B_i\subset M$. Let $\bar{B}_i$ denote the corresponding closed balls. Then 
$$
\bigcup_{i} \bar{B}_i
$$
is a compact subset of $M$ with finitely many connected components. Connecting these components by finitely many paths we get a compact connected 
subset $L\subset M$ containing $K$. Let $\cV$ be a finite cover of $L$ by open coordinate balls. Then the union $N$ of these balls is paracompact, 
connected and contains $K$. \end{proof} 

\medskip
Even though the manifold $N$ in this lemma is not canonical, we will use the notation $\hat K$ for $N$.

\begin{cor}\label{cor:CA2}
Suppose that $M$ is a connected manifold and $K_i\subset M, i\in \mathbb N$, is a countable collection of compact subsets. Then $\bigcup_i K_i$ is contained 
in an open paracompact connected subset $N\subset M$. 
\end{cor}
\begin{proof} Without loss of generality, we may assume that $K_i\subset K_{i+1}$ for every $i$. Consider the family $\{\hat K_i: i\in \mathbb N\}$. The union $N$ of these 
open connected subsets is again open and connected. At the same time, it is contained in 
$$
Y:=\bigcup_{i\in \N} cl(\hat K_i),
$$ 
a $\sigma$-compact subset of $M$. Then $Y$ is paracompact by Proposition \ref{prop:PA1}. Therefore,  $N\subset Y$ is paracompact as well. 
\end{proof}

\section{Proof of Theorem \ref{thm:main}}

 In order to prove Theorem \ref{thm:main} it suffices to consider connected conformal manifolds and prove that they are always metrizable. 
 More precisely, we will check that the complement to a closed coordinate ball in $M$ admits a conformally-natural metric. 
The construction of this metric mostly follows the work of Ferrand, \cite{Ferrand}. Hidden behind the construction is again Perron's method, but it is used differently 
from the standard proofs of Rado's theorem for Riemann surfaces. We let $Lip(M)$ denote the space of Lipschitz 
continuous functions on $M$. (Ferrand uses a slightly different functional space.) 
Lipschitz continuous functions are differentiable a.e. on $M$ and norms of their gradients are locally bounded and, hence, locally integrable.  
We let $Lip_p(M)$ denote the subspace consisting of functions in $Lip(M)$ with {\em paracompact support} and $Lip_c(M)$ the 
subspace of $Lip_p(M)$ consisting of functions with compact support. We will use the fact that for $f_1, f_2\in Lip(M)$, $\max(f_1,f_2)$ is also in $Lip(M)$. 

Recall that our manifold $M$ is $n$-dimensional. 
Given a function $f\in Lip(M)$ we define its {\em energy-density} $e_f(x)$ as   
$$
|\nabla f(x)|^n dV, x\in M,
$$
where the gradient, its norm and the volume density $dV$ are defined with respect to a local conformal metric (defined in a neighborhood of $x$).  
The definition of $e_f$ is independent of the choice of a local conformal metric and $e_f$ is locally integrable. 
Assuming that $f$ is in $Lip_p(M)$, we have the {\em $n$-energy integral} 
$$
I_M(f):= \int_M e_f\in [0,\infty]. 
$$
Note that this integral is well-defined since the support set of $e_f$ is paracompact and, hence, we can use a partition of unity on this support to define the integral. 

Let $C_1, C_2$ be two closed subsets of $M$. The {\em capacity of the pair $(C_1, C_2)$} is defined as
$$
Cap_M(C_1,C_2)=\inf_{f\in A(M,C_1,C_2)} I_M(f),
$$
where $A(M,C_1,C_2)$, the space of {\em admissible functions} with respect to $(C_1, C_2)$, consisting of functions $f\in Lip_p(M)$ such that $0\le f\le 1$, 
$f|_{C_1}\equiv 0$, $f|_{C_2}\equiv 1$. Thus, $0\le Cap_M(C_1,C_2)=Cap_M(C_2, C_1)\le \infty$. By the definition, capacity is a conformal invariant. 
It is also clear that if $N$ is an open subset of a conformal manifold $M$ (with the induced conformal structure), then for any pair of closed subsets 
$C_1, C_2\subset M$, we have
\begin{equation}\label{eq:monotonic}
Cap_N(C_1,C_2)\le Cap_M(C_1,C_2)
\end{equation}
as the restriction map sends  $A(M,C_1,C_2)$ to $A(N,C_1,C_2)$ and decreases the energy. Recall that a compact metrizable space is  called 
a {\em continuum} if it is connected. A {\em nondegenerate continuum} is one which has cardinality at least two (hence, cardinality of continuum). 
A compact subset of a manifold is metrizable, hence, every compact connected subset of a manifold is a continuum. 
We will need the following result from \cite[(1.3)]{Ferrand}: 

\begin{lem}\label{lem:L0}
Let $B=B(\mathbf 0, 1)\subset \R^n$ be the open unit ball equipped with some conformal structure (not necessarily the standard one). 
Then for every $\eps>0$ and $r>0$ there exists 
$f\in Lip_c(B)$ which is identically equal to $1$ on the ball $B(\mathbf 0, r)$ and satisfies $I_B(f)<\eps$.  
\end{lem}

From now on, we will assume that our manifold $M$ is connected.

\begin{lem}\label{lem:L2}
1. Suppose that $C_1, C_2$ are disjoint compact subsets of $M$. Then $Cap(C_1, C_2)<\infty$. 

2. Suppose that $C_1, C_2$ are nondegenerate continua in $M$. Then $Cap_M(C_1,C_2)>0$. 
\end{lem}
\begin{proof} 1. Take an open paracompact subset $U\subset M$ containing $C_1\cup C_2$ and a compactly supported smooth function $f\in A(U, C_1, C_2)$. Then 
extending $f$ by $0$ to the rest of $M$, we obtain an admissible function of finite energy.  

2. Let $N=\hat C\subset M$ be an open connected paracompact subset containing $C=C_1\cup C_2$. Thus, $Cap_N(C_1,C_2)\le Cap_M(C_1,C_2)$. 
But for connected paracompact conformal manifolds $N$ positivity of $Cap_N(C_1,C_2)$ is proven in \cite[(3.6)]{Ferrand}. Now the conclusion follows from the inequality 
\eqref{eq:monotonic}. \end{proof}

Following Ferrand, \cite{Ferrand}, we next define a certain pseudometric on $M$ associated canonically with the conformal structure of $M$. 
For a compact $K\subset M$ set
$$
Cap_M(K):= \inf_f I_M(f),
$$
where the infimum is taken over all functions $f\in Lip_c(M)$ which equal to $1$ on $K$. Such functions will be called {\em $K$-admissible}. Clearly, if $N\subset M$ 
is an open connected subset, then for every compact $K\subset M$,
$$
Cap_M(K)\le Cap_N(K). 
$$

\begin{lem}\label{lem:L3}
$Cap_M(K)<\infty$ for every compact $K\subset M$. 
\end{lem}
\begin{proof} Take a relatively compact open subset $N\subset M$ containing $K$. Then pick a function $f\in C^1_c(N)$ which is identically $1$ 
on $K$ and extend it by zero to the rest of $M$. Thus, 
$Cap_M(K)\le I_M(f)<\infty$.  \end{proof}

\begin{lem}\label{lem:L4}
If $K_1\subset K_2$ are compacts in $M$, then 

1. $Cap_M(K_1)\le Cap_M(K_2)$,  

2. $Cap_M(K_1\cup K_2)\le Cap_M(K_1)+ Cap_M(K_2)$.  
\end{lem}
\begin{proof} The first part is immediate. To prove the second part, take $K_i$-admissible functions $f_i, i=1,2$, on $M$ 
and set $f:= \max(f_1, f_2)$. Then $f$ is $K$-admissible for $K=K_1\cup K_2$. Set
$$
A_1:= \{x: f_1(x)> f_2(x)\}, A_2:= \{x: f_2(x)> f_1(x)\}, A_0:= \{x: f_1(x)=f_2(x)\}. 
$$
Thus, $M= A_1\sqcup A_2\sqcup A_0$. We have
$$
I_M(f)= I_{A_1}(f_1) + I_{A_2}(f_2) + \int_{A_0} e_f. 
$$
Clearly,
$$
I_{A_1}(f_1) + I_{A_2}(f_2)= \int_{A_1\cup A_2} (e_{f_1} + e_{f_2})= \int_{A_1\cup A_2} e_f. 
$$
It remains to analyze the integrals of energy-densities over $A_0$ (note that $A_0$ can have positive measure). 
Take a point $x_0\in A_0$, and fix a local conformal metric $g$ on an open coordinate ball $B\subset M$ centered at $x_0$. Then, with respect to this metric, 
$$
|\nabla f(x)|\le \max( |\nabla f_1(x)|, |\nabla f_2(x)|)\le  |\nabla f_1(x)|+ |\nabla f_2(x)|, x\in B\cap A_0.$$ 
Hence (denoting $dV$ the volume density of $g$), 
$$
\int_{B\cap A_0} e_f \le \int_{B\cap A_0}  (|\nabla f_1(x)|+ |\nabla f_2(x)|)dV=  \int_{B\cap A_0} (e_{f_1} + e_{f_2}). 
$$
Therefore, 
$$
\int_{A_0} e_f\le \int_{A_0} (e_{f_1} + e_{f_2}). 
$$
Lemma follows.  
\end{proof}

\begin{lem}\label{lem:L5}
Let $N\subset M$ be the complement to a closed coordinate ball $\bar{B}\subset M$. Then for every nondegenerate continuum 
$K\subset N$ we have $Cap_N(K)>0$. 
\end{lem}
\begin{proof} Take a $K$-admissible function $f\in Lip_c(N)$ and extend it by $0$ to $\bar{B}$. We get a function $u\in A(M, \bar{B}, K)$ and $I_N(f)=I_M(u)$. 
By Lemma \ref{lem:L2} (Part 2), there exists $r>0$ independent of $f$, such that $I_M(u)\ge r$.  Hence, $I_N(f)\ge r>0$ and lemma follows. 
\end{proof}

\begin{lem}\label{lem:L6}
For every singleton $K=\{x\}\subset M$, $Cap_M(K)=0$. 
\end{lem}
\begin{proof} Let ${B}$ be an open coordinate ball in $M$ centered at $x$. Then, according to Lemma \ref{lem:L0},  there exists a sequence of 
smooth functions $f_i\in Lip_c(B)$ which are all equal to $1$ at $x$ and
$$
\lim_{i\to\infty} I_B(f_i)=0. 
$$
Extending functions $f_i$ by $0$ to the rest of $M$, we conclude that $Cap_M(K)=0$. 
\end{proof}

\begin{definition}
For points $x, y\in M$ define $\mu_M(x,y):= \inf_C Cap_M(C)$, where the infimum is taken over all continua $C\subset M$ containing $\{x, y\}$. 
\end{definition}

\begin{lem}\label{lem:L7}
The function $\mu_M$ is a finite pseudometric on $M$. 
\end{lem} 
\begin{proof} Symmetry of $\mu_M$ is clear. Since $M$ is a connected manifold, it is path-connected; 
hence, every two points $x, y\in M$ belong to a continuum $C\subset M$. 
Lemma \ref{lem:L3} then implies that $\mu_M(x,y)\le Cap_M(C)<\infty$, 
hence, $\mu_M$ is finite.  
The triangle inequality follows from Lemma \ref{lem:L4} (Part 2).  Lemma \ref{lem:L6} implies that $\mu_M(x,x)=0$, 
since we can take $C=\{x\}$ as our continuum containing $\{x\}$.  
\end{proof}

\begin{definition}
A conformal manifold $M$ is said to be of {\em Class I} if the pseudometric $\mu_M$ is not a metric and is said to be of {\em Class II} otherwise. 
\end{definition}

\begin{prop}\label{prop:P1}
The following are equivalent:

1. For some pair of distinct points $x, y\in M$, $\mu_M(x,y)=0$, i.e. $M$ is of Class I. 

2. For all pairs of points $x, y\in M$, $\mu_M(x,y)=0$. 

3. For every continuum $C\subset M$, $Cap_M(C)=0$. 

4. There exists a nondegenerate continuum $C\subset M$, such that $Cap_M(C)=0$. 
\end{prop}
\begin{proof} The only part which is not obvious is the implication (1)$\Rightarrow$(3). This implication is proven in \cite[(6.8)]{Ferrand} for paracompact manifolds. 
One way to argue would be to adapt her proof to the general case where manifolds are not assumed to be paracompact. 
Instead, we reduce the general case to the paracompact one. Since 
$\mu_M(x,y)=0$, there exists a sequence of nondegenerate continua $C_i\subset M$ (containing $\{x,y\}$) and 
$C_i$-admissible functions $f_i\in Lip_c(M)$ such that $I_M(f_i)< 1/i$. Let $K_i$ denote 
the (compact) support set of $f_i$, hence, $C_i\subset K_i$. The union
$$
C\cup \bigcup_i K_i 
$$ 
is contained in a paracompact open connected subset $N\subset M$, see Corollary \ref{cor:CA2}. 
Thus, $\mu_N(x,y)=0$ (since each $f_i$ restricts to a compactly supported function on $N$). 
But this implies that $Cap_N(C)=0$ as proven by Ferrand,   
\cite[(6.8)]{Ferrand}. Since $Cap_M(C)\le Cap_N(C)=0$, we conclude that $Cap_M(C)=0$ as well. 
\end{proof}

\begin{cor}\label{cor:C1}
(See \cite[Example 6.9(b)]{Ferrand}.) Let $K=\bar{B}\subset M$ be a closed coordinate ball. Then the manifold $N:= M- K$ is of Class II. 
\end{cor}
\begin{proof} So far, we have not used the assumption that $M$ has dimension $>1$ (except, indirectly, in Lemma \ref{lem:L2}). 
We will use it now explicitly and observe that due to this dimension assumption, 
the manifold $N$ is connected. Suppose that $\mu_N$ is not a metric. Then, by Proposition \ref{prop:P1}, there exists a nondegenerate 
continuum $C\subset N$ such that $Cap_N(C)=0$, which implies that 
$Cap_M(C, K)=0$. But this contradicts Part 2 of Lemma \ref{lem:L2}. 
\end{proof}

\begin{prop}
If $M$ is of Class II, then the metric $\mu_M$ metrizes $M$ as a topological space. 
\end{prop}
\begin{proof} 1. We first prove that the function $\mu_M: M^2\to \R$ is continuous at the diagonal, i.e. if $x_i, y_i$ are sequences in $M$ converging to the same point 
$z\in M$, then $\mu_M(x_i, y_i)\to 0$. Fix an open coordinate  ball $B\subset M$ centered at $z$. Without loss of generality, 
$x_i, y_i\in B$ for all $i$. Then,  
$$
\mu_B(x_i,y_i)\to 0,
$$
see Lemma \ref{lem:L0}. Since $\mu_M(x_i,y_i)\le \mu_B(x_i,y_i)$, we get 
$$
\lim_{i\to\infty} \mu_M(x_i,y_i)= 0= \mu_M(z,z). 
$$ 

2. Let us check continuity of $\mu_M$ at general pairs $(x,y)\in M^2$. Consider sequences $x_i\to x, y_i\to y$ in $M$. 
Then, by the triangle inequality for $\mu_M$ and Part 1 of the proof:
$$
\mu_M(x,y)\le \liminf_{i\to\infty} (\mu_M(x, x_i) + \mu_M(x_i, y_i) + \mu_M(y_i,y))\le   \liminf_{i\to\infty}  \mu_M(x_i, y_i). 
$$ 
Similarly, 
$$
\limsup_{i\to\infty} \mu_M(x_i, y_i)\le \mu_M(x,y). 
$$
It follows that $\mu_M$ is continuous at $(x,y)$. Thus, the manifold topology of $M$ is stronger than the metric topology induced by $\mu_M$. 

3. Let us prove that the metric topology is stronger than the manifold topology. 
Let $B\subset M$ be an open coordinate ball centered at $x\in M$. We will show that there exists $r>0$ such that for every $z\in M- {B}$, $\mu_M(x,z)\ge r$. 
Indeed, since the function $h(y)=\mu_M(x,y)$ is continuous and $S=\partial B$ is compact, the function $h$ has positive minimum $r>0$  on $S$. Take $z\in M- B$ 
and a continuum $C\subset M$ connecting $x$ and $z$. This continuum has to intersect $S$ at some point $y$. Therefore, $r\le \mu(x,y)\le Cap_M(C)$ and, thus,  
$\mu_M(x,z)\ge r$. Proposition follows. 
\end{proof}

We can now finish the proof of the theorem. Consider a closed coordinate ball $\bar{B}\subset M$ which we will identify with the 
closed unit ball $\bar{B}(\mathbf 0, 1)\subset \R^n$. Take two smaller closed  balls, 
$$
\bar{B}_1=\bar{B}(\mathbf 0, 1/4)\subset  \bar{B}_2=\bar{B}(\mathbf 0, 1/2)\subset \bar{B}(\mathbf 0, 1)=\bar{B}.$$
By Corollary \ref{cor:C1}, the complement $N= M - \bar{B}_1$  is metrizable 
(by the metric $\mu_N$). Let $B_2$ be the interior of $\bar{B}_2$; it is an open coordinate ball in $M$.  
Hence, the closed subset $X=M- {B}_2\subset M$ is metrizable as well (by the restriction of 
the metric $\mu_N$). Therefore, $M$ is the union of two closed paracompact subsets: $X$ and $\bar{B}$ (the latter is even compact). 
But a topological space which is the union of finitely many closed paracompact subsets is itself paracompact. 
Paracompactness of $M$ follows.

\bibliographystyle{amsalpha}

\bibliography{references}

\providecommand{\bysame}{\leavevmode\hbox to3em{\hrulefill}\thinspace}
\providecommand{\MR}{\relax\ifhmode\unskip\space\fi MR }
\providecommand{\MRhref}[2]{%
  \href{http://www.ams.org/mathscinet-getitem?mr=#1}{#2}
}
\providecommand{\href}[2]{#2}
\begin{thebibliography}{Bou89}

\bibitem[Bou89]{Bourbaki}
N.~Bourbaki, \emph{General topology}, Elements of mathematics, Springer-Verlag,
  Berlin, 1989.

\bibitem[CR53]{CR}
E.~Calabi and M.~Rosenlicht, \emph{Complex analytic manifolds without countable
  base}, Proceedings of the American Mathematical Society \textbf{4} (1953),
  335--340.

\bibitem[Die44]{D}
J.~Dieudonn\'e, \emph{Une g\'enéralisation des espaces compacts}, Journal de
  Math\'ematiques Pures et Appliqu\'ees, Neuvi\'eme S\'erie \textbf{23} (1944),
  65--76.

\bibitem[Eng89]{Engelking}
R.~Engelking, \emph{General topology}, Sigma series in pure mathematics,
  vol.~6, Heldermann, Berlin, 1989.

\bibitem[Fer96]{Ferrand}
J.~Ferrand, \emph{Conformal capacities and conformally invariant functions on
  riemannian manifolds}, Geometriae Dedicata \textbf{61} (1996), 103--120.

\bibitem[For81]{Forster}
O.~Forster, \emph{Lectures on {R}iemann surfaces}, Springer {V}erlag, New York,
  1981.

\end{thebibliography}

\noindent
Department of Mathematics, \\
University of California, Davis,\\ 
 CA 95616, USA \\
 kapovich@math.ucdavis.edu

\end{document}